\documentclass[11pt]{amsart}
\usepackage{fullpage}
\newcommand{\volumetitle}[1]{}
\newcommand{\emsaffil}[2]{}
\newcommand{\emsauthor}[3]{}
\newcommand{\titlemark}[1]{}
\newenvironment{ack}{\section*{Acknowledgements}}{}
\newenvironment{funding}{}{}
\renewcommand{\paragraph}[1]{
  \subsection*{#1.}}
\usepackage{booktabs}
\usepackage{cite}

\author{Roland Bauerschmidt}
\address{University of Cambridge, Statistical Laboratory, DPMMS}
\email{rb812@cam.ac.uk}

\author{Tyler Helmuth}
\address{Durham University, Department of Mathematical Sciences}
\email{tyler.helmuth@durham.ac.uk}

\usepackage{thmtools, thm-restate}
\declaretheorem[numberwithin=section]{theorem}

\numberwithin{equation}{section}

\begin{document}

\volumetitle{ICM 2022} 

\title{Spin systems with hyperbolic symmetry: a survey}
\titlemark{Spin systems with hyperbolic symmetry}

\emsauthor{1}{Roland Bauerschmidt}{R.~Bauerschmidt}
\emsauthor{2}{Tyler Helmuth}{T.~Helmuth}

\emsaffil{1}{DPMMS, University of Cambridge, CB3 0WB, UK \email{rb812@cam.ac.uk}}
\emsaffil{2}{Mathematical Sciences,
Durham University, DH1 3LE, UK \email{tyler.helmuth@durham.ac.uk}}


\begin{abstract}
  Spin systems with hyperbolic symmetry originated as simplified
  models for the Anderson metal--insulator transition, and were
  subsequently found to exactly describe probabilistic models of
  linearly reinforced walks and random forests.  In this survey we
  introduce these models, discuss their origins and main features,
  some existing tools available for their study, recent probabilistic
  results, and relations to other well-studied probabilistic models.
  Along the way we discuss some of the (many) open questions that
  remain.
\end{abstract}

\maketitle

\section{Introduction}

Classical spin systems with spherical symmetry, such as the Ising and
classical Heisenberg models, are basic models for magnetism and have
been studied extensively over the last century. It is well-understood
that the associated symmetry groups play an important role,
particularly for the critical and low-temperature behaviour of these
models.  For example, the discrete $\mathbb{Z}_{2}$ symmetry of the
Ising model is spontaneously broken at low temperatures, and in this
phase truncated correlations decay exponentially. For models with
continuous $O(n)$ symmetries, $n\geq 2$, low temperature truncated
correlations instead decay polynomially, a reflection of the fact that
the symmetry is spontaneously broken to an $O(n-1)$ symmetry.

Spin systems with hyperbolic symmetry groups are also studied in
condensed matter physics, primarily because of their relevance for 
the Anderson delocalisation--localisation (metal--insulator) transition 
of random Schr\"odinger operators and related random matrix models 
\cite{10.1007/BF01319839,MR575503,MR708812}.
A rigorous analysis of the Anderson transition
remains an outstanding challenge; see
Section~\ref{sec:random-band-matrices}. 
The essential physical phenomena of the Anderson transition are
expected to be captured by the more tractable $\mathbb{H}^{2|2}$
model, a simplified spin system with hyperbolic
symmetry~\cite{MR1134935}.  Surprisingly, the $\mathbb{H}^{2|2}$ model
and its natural generalisations are intimately connected to
probabilistic lattice models.  The $\mathbb{H}^{2}$ and
$\mathbb{H}^{2|2}$ models, motivated by the Anderson
transition~\cite{MR2728731,MR2104878}, are exactly related to
(linearly) edge-reinforced random walks and vertex-reinforced jump
processes, introduced independently in the probability literature in
the 1980s~\cite{MR1008047} and early
2000s~\cite{MR1900324}. 
A similar connection exists between the related $\mathbb{H}^{0|2}$
model and random forests~\cite{MR2110547,MR4218682}; random forests
arose earlier (for example) in connection with the Fortuin--Kasteleyn
random cluster model~\cite{MR2243761}. The connections between
hyperbolic spin systems and probabilistic phenomena are the main topic
of this survey.

More specifically, this survey focuses on probabilistic results in
line with the original physical motivation for studying hyperbolic
spin systems. In particular, we focus on results for $\mathbb{Z}^{d}$
(and its finite approximations) for $d\geq 2$.  Our perspective is
that a central role is played by the continuous symmetry groups of the
spin systems. There are other perspectives available, notably that of
Bayesian statistics. While the latter perspective has played a role in
important results, e.g.~\cite{MR3189433,2102.08988}, and has found use
in statistical contexts~\cite{MR2816340,MR2278358,MR3534357}, we will
not mention it further. Similarly, there are many related works we
cannot discuss; fortunately many of these are discussed in recent
surveys on closely related topics~\cite{MR2282181,
  MR3204347,MR2953867,MR3469136}.

To set the stage, the remainder of this introduction recalls the
\emph{magic formula} for edge-reinforced random walk that lead to the
discovery of the connections discussed in this survey. Readers
familiar with the magic formula may wish to jump to
Section~\ref{sec:hyp}, where we introduce hyperbolic spin systems, or
to Section~\ref{sec:random-band-matrices}, which discusses the
physical background. The probabilistic representations and results for
reinforced random walks and random forests are discussed in
Sections~\ref{sec:walks} and~\ref{sec:forests}, respectively, along
with questions for the future.

\paragraph{Magic Formula for Edge-Reinforced Random Walk}

Fix $\alpha>0$, a graph $G=(\Lambda,E)$, and an initial vertex
$0\in \Lambda$. \emph{Edge-reinforced random walk (ERRW)} with
$X_{0}=0$ and initial weights $\alpha$ is the stochastic process
$(X_{n})_{n\geq 0}$ with transitions
\begin{equation}
  \label{eq:ERRW}
  \mathbb{P}^{\text{ERRW}(\alpha)}_{0}[X_{n+1}=j | (X_{m})_{m\leq n}, X_{n}=i] =
  \frac{(\alpha+L_n^{ij})1_{ij\in E}}{\sum_{k:ik\in E}(\alpha+L_n^{ik})},
\end{equation}
where $L_{n}^{ij}$ is the number of times the edge $ij$ has been
crossed up to time $n$ (in either direction). 
The transition
rates change rapidly if $\alpha$ is small, and hence this is called the
\emph{strong reinforcement} regime.
\emph{Weak reinforcement} refers to $\alpha$ being large.
The definition can be generalised to edge dependent weights $\alpha=(\alpha_{ij})$
in a straightforward manner.

Some intuition about ERRW can be gained by considering the case when
$G$ is a path on three vertices. Call one edge blue, one edge red, and
start an ERRW at the middle vertex. If $\alpha=2$ then the law of the
vector $\frac{1}{2}L_{2n}$ of (half of) the number of crossings of the
edges at time $2n$ is the law of a \emph{P\'{o}lya urn}. P\'{o}lya's
urn is the process that starts with an urn containing one red and one
blue ball, and then sequentially draws a ball and replaces it with two
balls, both of the same colour as the drawn ball. The fundamental fact
about P\'{o}lya's urn is that $\frac{1}{2n}L_{2n}$ converges to
$(U,1-U)$ where $U$ is a uniform random variable on $[0,1]$, i.e., the
fraction of crossings of the blue edge is uniform. This can be proven
by induction. Note that for an ordinary simple random walk this limit
would be deterministic. \emph{A priori} it is hard to predict how ERRW
behaves on more complicated graphs. For example, is ERRW transient if
simple random walk is transient? Does the answer depend on $\alpha$?

It turns out that the connection to P\'{o}lya's urn has a far reaching
generalisation. The theory of partial exchangeability guarantees that
ERRW is a random walk in random environment~\cite{MR577313}. A
consequence is that $\frac{1}{n} L_{n}$ has a distributional limit: it
is the law of the random environment. Coppersmith and Diaconis
discovered that one can give an explicit formula for the limiting law
on \emph{any} finite graph. It is surprising that an explicit formula
can be obtained; this explains why it has been termed the \emph{magic
  formula}, see \cite{MR1832379,MR2441859}.

To precisely formulate this result, recall an \emph{environment} is a
set of conductances $C\colon E\to [0,\infty)$ with
$\sum_{ij}C_{ij}=1$. Write $C_{ij}$ for the conductance of the edge
$\{i,j\}$ and $C_{i} = \sum_{j}C_{ij}$. Associated to $C$ is a
reversible Markov chain (simple random walk) with transition
probabilities $C_{ij}/C_{i}$ whose law we denote by
$\mathbb{P}_0^{\text{SRW}(C)}$ when started from $0$.

\begin{theorem}[Magic formula for ERRW]
  \label{thm:errw-magic}
  Let $G=(\Lambda,E)$ be finite. Edge-reinforced random walk with
  $X_{0}=0$ and initial weights $\alpha= (\alpha_{ij})$ is a random
  walk in random environment:
\begin{equation}
  \label{eq:RWRE}
  \mathbb{P}^{\text{ERRW}(\alpha)}_0[\cdot] = \int \mathbb{P}^{\text{SRW}(C)}_0[\cdot] \,d\mu_\alpha(C).
\end{equation}
  The environment $\mu_\alpha$ has density proportional to
  \begin{equation}
    \label{eq:errw-magic}
    C_{0}^{\frac{1}{2}}
    \frac{\prod_{ij\in
        E}C_{ij}^{\alpha_{ij}-1}}{
      \prod_{i\in\Lambda}
      C_{i}^{\frac{1}{2}(\alpha_{i}+1)}} \sqrt{\det C}
  \end{equation}
  with respect to Lebesgue measure on the unit simplex in
  $[0,\infty)^{E}$, and where
  $\alpha_{i}=\sum_{j}\alpha_{ij}$. 
\end{theorem}
Sabot and Tarr\`{e}s showed how to relate the
density~\eqref{eq:errw-magic} to the $\mathbb{H}^{2|2}$ model that we
will introduce in the next section. This enabled them to leverage
powerful results of Disertori, Spencer, and Zirnbauer to establish the
existence of a recurrence/transience phase transition for ERRW on
$\mathbb{Z}^{d}$ for $d\geq 3$, see Section~\ref{sec:walks}.
In Section~\ref{sec:forests} we show that connection probabilities in
the arboreal gas, a stochastic-geometric model of random forests, can
be written in a form very similar to the magic formula.  The
derivation of this connection probability formula was inspired
by~\cite{MR4021254,MR4255180}, which (at least partially) revealed the
inner workings of the magic formula: horospherical coordinates
(hyperbolic symmetry) and supersymmetric localisation.

\section{Hyperbolic spin systems}
\label{sec:hyp}

This section introduces the hyperbolic spin systems that we will
discuss, briefly explains their characteristic symmetries, and
discusses how these symmetries manifest themselves if spontaneous
symmetry breaking occurs. For precise definitions of the
Grassmann and Berezin integrals that are used see, e.g., \cite[Appendix~A]{MR4255180}.

\paragraph{The $\mathbb{H}^{2|0}$ model}
The $\mathbb{H}^2 = \mathbb{H}^{2|0}$ model is defined as follows. We
consider the hyperbolic plane $\mathbb{H}^2$ realised as
$\mathbb{H}^2 = \{u = (x,y,z) \in \mathbb{R}^3: x^2+y^2-z^2 = -1, z
>0\}$ and equipped with the Minkowski inner product
$u \cdot u' = xx'+yy'-zz'$. For a finite graph $G=(\Lambda,E)$, we consider
one spin $u_i \in \mathbb{H}^2$ per vertex $i \in \Lambda$ and define
the action
\begin{equation}
  H_{\beta,h}(u) = \frac{\beta}{2} \sum_{ij \in E} (u_i-u_j)\cdot (u_i- u_j) + h \sum_{i\in\Lambda} z_i.
\end{equation}
The action also has a straightforward generalisation to edge- and vertex-dependent weights $\beta=(\beta_{ij})$
and $h=(h_i)$, and we will sometimes consider this case.
For $\beta>0$ and $h=0$, the minimisers
of $H_{\beta,0}$ are constant configurations
$u_i = u_j$ for all $i,j\in\Lambda$.
For $h>0$, the unique minimiser is $u_i = (0,0,1)$ for all $i$.
The $\mathbb{H}^{2}$ model is the
probability measure on spin configurations whose expectation is given,
for bounded $F\colon (\mathbb{H}^2)^\Lambda \to \mathbb{R}$, by
\begin{equation}
  \label{eq:H2hyp}
  \langle F\rangle_{\beta,h}
  =
  \frac{1}{Z_{\beta,h}}
  \int_{(\mathbb{H}^{2})^\Lambda}
  \prod_{i\in\Lambda} du_i \,
  F(u) \, e^{-H_{\beta,h}(u)}
\end{equation}
where $du_i$ stands for the Haar measure on $\mathbb{H}^2$ and
$Z_{\beta,h}$ is a normalisation.
Parametrising $u_i\in \mathbb{H}^2$ by $(x_i,y_i) \in \mathbb{R}^2$
with $z_i = \sqrt{1+x_i^2+y_i^2}$, we can explicitly
rewrite~\eqref{eq:H2hyp} as
\begin{equation}
  \label{eq:H2flat}
  \langle F\rangle_{\beta,h}
  =
  \frac{1}{Z_{\beta,h}}
  \int_{(\mathbb{R}^{2})^\Lambda}
  \prod_{i\in\Lambda} \frac{dx_i \, dy_i}{z_i} \,
  F(u) \, e^{-H_{\beta,h}(u)}
  .
\end{equation}
The expectation is only normalisable if $h>0$ (or more generally
$h_i>0$ for some vertex $i$) due to the non-compactness of $\mathbb{H}^{2}$.
It is useful to construct a version with $h=0$ in which the field is
fixed (pinned) at some distinguished vertex $0$.
We denote the \emph{pinned expectation} with pinning
$u_0=(0,0,1)$ by $\langle\cdot\rangle_\beta^0$.

\paragraph{The $\mathbb{H}^{0|2}$ model}

Now we consider the Grassmann algebra $\Omega_\Lambda$ generated by two generators $\xi_i$ and $\eta_i$
per vertex $i \in\Lambda$ and set
\begin{equation}
  z_i = \sqrt{1-2\xi_i\eta_i} = 1-\xi_i\eta_i,
\end{equation}
and unite these into the formal supervector $u_i = (\xi_i,\eta_i,z_i)$. Thus $u_i$ has two odd (anticommuting) components and one even (commuting) component.
We define $u_i\cdot u_j = -\xi_i\eta_j-\xi_j\eta_i -z_iz_j$, which is again an element of $\Omega_\Lambda$.
These definitions are such that $u_i\cdot u_i = -1$, as in the case of
$\mathbb{H}^2$ spins. Define 
\begin{equation}
  H_{\beta,h}
  = \frac{\beta}{2}  \sum_{ij \in E} (u_i-u_j)\cdot (u_i- u_j) + h \sum_{i\in\Lambda} z_i.
\end{equation}
For $F$ a polynomial in the $\xi_i$ and $\eta_i$ set
\begin{equation}
  \label{eq:H02}
  \langle F\rangle_{\beta,h}
  =
  \frac{1}{Z_{\beta,h}}
  \int \left(  \prod_{i\in\Lambda} \partial_{\eta_i} \partial_{\xi_i} \frac{1}{z_i} \right)
  F \, e^{-H_{\beta,h}}
  ,
\end{equation}
where $\int \prod_{i\in\Lambda} \partial_{\eta_i} \partial_{\xi_i}$ stands for the Grassmann integral,
i.e., the top coefficient of the element of the Grassmann algebra to its right.
For example,
\begin{equation}
  \label{eq:H02ex}
  \int \partial_{\xi}\partial_\eta \, e^{-\xi\eta} =
  \int \partial_{\xi}\partial_\eta (1-\xi\eta) =
  \int \partial_{\xi}\partial_\eta \, \eta\xi =  1.
\end{equation}
In~\eqref{eq:H02} and~\eqref{eq:H02ex}
we have used the convention that smooth functions of commuting
elements of the algebra are defined by Taylor expansion. By nilpotency
the expansion is finite, i.e., a polynomial.  The $\mathbb{H}^{0|2}$
model is the expectation~\eqref{eq:H02}; while this is not a
probabilistic expectation, we will soon see that it often carries
probabilistic interpretations. Generalisations to edge- and
vertex-dependent weights and pinning are straightforward. 

\paragraph{The $\mathbb{H}^{2|2}$ model}

The $\mathbb{H}^{2|2}$ model is defined as the $\mathbb{H}^{0|2}$ model was,
but now beginning with three commuting components $x_{i},y_{i},z_{i}$.
Formally, this means the real coefficients of the Grassmann algebra
$\Omega_{\Lambda}$ of the previous section are replaced by smooth
functions of $x_{i}$ and $y_{i}$. To each vertex $i$ we associate a
formal supervector 
$u_{i} = (x_{i},y_{i},\xi_{i},\eta_{i},z_{i})$, where $x_{i}$ and
$y_{i}$ are commuting, $\xi_{i}$ and $\eta_{i}$ are generators of a
Grassmann algebra, and
\begin{equation}
  \label{e:H22z}
  z_{i} = \sqrt{1+x_{i}^{2}+y_{i}^{2}-2\xi_{i}\eta_{i}} =
  \sqrt{1+x_{i}^{2}+y_{i}^{2}} - \frac{\xi_{i}\eta_{i}}{\sqrt{1+x_{i}^{2}+y_{i}^{2}}}.
\end{equation}
As for $\mathbb{H}^{0|2}$, smooth functions of commuting elements of
this algebra are defined by Taylor expansion, with the expansion now
performed about $(x_{i},y_{i})\in \mathbb{R}^{2\Lambda}$; the second
equality of~\eqref{e:H22z} is an example.

The definition \eqref{e:H22z} ensures that $z_{i}$ has positive
degree zero part,
and that $u_{i}\cdot u_{i}=-1$ for the super inner product
$u_{i}\cdot u_{j} = x_{i}x_{j}+y_{i}y_{j}-\xi_{i}\eta_{j}-\xi_{j}\eta_{i}-z_{i}z_{j}$.
As previously, we define
\begin{equation}
  \label{eq:H22H}
  H_{\beta,h}(u) = \frac{\beta}{2}  \sum_{ij\in E}(u_{i}-u_{j})\cdot
  (u_{i}-u_{j}) + h\sum_{i\in\Lambda}z_{i},
\end{equation}
and the associated expectation
\begin{equation}
  \label{eq:H22E}
  \langle F \rangle_{\beta,h} =
  \frac{1}{(2\pi)^{|\Lambda|}}
  \int
  \left(\prod_{i\in\Lambda}dx_{i}dy_{i}\partial_{\eta_{i}}\partial_{\xi_{i}}\frac{1}{z_{i}}\right)
  F e^{-H_{\beta,h}}.
\end{equation}
This integral combines ordinary integration and Grassmann integration
and is an instance of the Berezin integral, sometimes call a super integral
\cite{MR914369}.  One computes the Grassmann integral to obtain the
top coefficient of the element of the Grassmann algebra; this is a
smooth function on $\mathbb{R}^{2\Lambda}$. One then computes the
ordinary Lebesgue integral of this function.  Again the generalisation
to edge- and vertex-dependent weights and pinning is straightforward.

Note that~\eqref{eq:H22E} does not have a normalising factor
as in the definitions of the $\mathbb{H}^{2}$ and $\mathbb{H}^{0|2}$
models, aside from the factor $(2\pi)^{-|\Lambda|}$ that does not
depend on the weights.  Nonetheless the expectation is normalised:
$\langle 1 \rangle_{\beta,h}=1$ if $h>0$.  This is due to an internal
supersymmetry in the model, which implies
$Z_{\beta,h}=(2\pi)^{|\Lambda|}$. More generally, this supersymmetry
implies a powerful localisation principle first used in this context
in~\cite{MR2728731}.

\begin{theorem}[SUSY localisation for $\mathbb{H}^{2|2}$]
  \label{thm:susyloc}
  For $F\colon
  \mathbb{R}^{\Lambda}\times\mathbb{R}^{\Lambda\times\Lambda}\to
  \mathbb{R}$ smooth and with sufficient decay,
  and for all edge- and vertex-dependent weights $\beta=(\beta_{ij})$ and $h=(h_i)$ with some $h_i>0$,
  \begin{equation}
    \label{eq:susyloc}
    \langle F((z_i), (u_{i}\cdot u_{j})) \rangle_{\beta,h} = F(1,-1).
  \end{equation}
\end{theorem}

On the right-hand side of~\eqref{eq:susyloc}, $1$ stands for the vector in
$\mathbb{R}^\Lambda$ with all entries equal to $1$, and similarly for
$-1$.  For example, $\langle z_{i}\rangle_{\beta,h}=1$ and
$\langle u_i\cdot u_j \rangle_{\beta,h}=-1$.

\paragraph{Beyond: $\mathbb{H}^{n|2m}$}

There is a natural generalisation of the above models to the broader
class of $\mathbb{H}^{n|2m}$ models with $n+1$ commuting coordinates
and $2m$ anticommuting coordinates.  Generalising
Theorem~\ref{thm:susyloc}, there is an exact correspondence between
observables of the $\mathbb{H}^{n|2m}$ and $\mathbb{H}^{n+2|2m+2}$
models, see \cite[Section~2]{MR4218682}. 
For developments when
$n=0$, see~\cite{1912.05817}.

\paragraph{Symmetries}
The $\mathbb{H}^{n|2m}$ models have continuous symmetries which are
analogues of the rotations of the $O(n)$ models.  For example, for the
hyperbolic plane $\mathbb{H}^2$, these symmetries are Lorentz boosts
and rotations.  The infinitesimal generator of Lorentz boosts in the
$xz$-plane is the linear differential operator $T$ acting as
\begin{equation}
  \label{eq:Lorentz}
  Tz=x, \quad Tx=z,\quad Ty=0.
\end{equation}
If $\mathbb{H}^{2}$ is parametrised by $(x,y)\in \mathbb{R}^{2}$, then
$T = z\partial_{x}$.  For the hyperbolic sigma models, there is an
infinitesimal boost $T_{i}=z_{i}\partial_{x_{i}}$ at each vertex $i$.
Haar measure on $\mathbb{H}^{2}$ and the action $H_{\beta,0}$ with
$h=0$ are invariant under these symmetries, i.e.,
$\sum_i T_i H_{\beta,0}=0$.  Analogous symmetries exist for
$\mathbb{H}^{0|2}$ and $\mathbb{H}^{2|2}$.  If $h>0$ then
$\sum_i T_i H_{\beta,h}\neq 0$, and the symmetries are said to be
explicitly broken by the external field.  Important consequences of
these symmetries are Ward identities. For example, when $n>0$ (such as
for the $\mathbb{H}^2$ and $\mathbb{H}^{2|2}$ models), for $h>0$,
\begin{equation} \label{e:ward1}
  \frac{\langle z_i \rangle_{\beta,h}}{h} = \sum_j \langle x_i x_j \rangle_{\beta,h},
\end{equation}
and when $m>0$ (such as for the $\mathbb{H}^{2|2}$ and $\mathbb{H}^{0|2}$ models),
\begin{equation} \label{e:ward2}
  \frac{\langle z_i \rangle_{\beta,h}}{h} = \sum_j \langle \xi_i \eta_j \rangle_{\beta,h}.
\end{equation}
Here $x_i$ and $(\xi_i,\eta_i)$ stand for an even (bosonic)
coordinate and pair of odd (fermionic) coordinates
when $n,m>1$, respectively. The
proofs of these identities boil down to integration by parts, see
e.g., \cite[Appendix~B]{MR2728731} or~\cite[Lemma~2.3]{MR4218682}.

\paragraph{Spontaneous symmetry breaking}

For the $\mathbb{H}^{2}$ and $\mathbb{H}^{2|2}$ models on a fixed finite graph, it is a consequence of the non-compactness of the hyperbolic symmetry
that, for example, $\langle x_0^2 \rangle_{\beta,h}$ diverges as $h\downarrow 0$. Similarly, for the $\mathbb{H}^{0|2}$ model on a finite graph,
symmetry implies that $\langle z_0\rangle_{\beta,h}$ tends to $0$ as $h\downarrow 0$. One of the main questions of statistical physics is whether
a symmetry survives in the infinite volume limit, or if it is \emph{spontaneously broken}.
To make this precise, it is convenient to consider a
finite volume criterion for this question. Consider a sequence of finite graphs $\Lambda$ that approximate $\mathbb{Z}^d$ in a suitable way
(denoted $\Lambda\to \mathbb{Z}^d$),
and let $\langle \cdot \rangle_{\beta,h}$ be the corresponding finite volume expectations.
For the  $\mathbb{H}^{2}$ and $\mathbb{H}^{2|2}$ models, there is spontaneous symmetry breaking (SSB) for a given $\beta$ if
\begin{equation} \label{e:SSBH2}
  \lim_{h\downarrow 0} \lim_{\Lambda \to \mathbb{Z}^d} \langle x_0^2\rangle_{\beta,h} < \infty,
\end{equation}
and similarly for the $\mathbb{H}^{0|2}$ model there is SSB if
\begin{equation} \label{e:SSBH02}
  \lim_{h\downarrow 0} \lim_{\Lambda\to \mathbb{Z}^d} \langle z_0\rangle_{\beta,h} >0.
\end{equation}
These notions can be understood by noticing that when the two limits
are exchanged the inequalities do not hold: in finite volume the $h=0$
symmetries are restored in the $h\downarrow 0$ limit, while they are
not in infinite volume if SSB occurs.  There are other notions of SSB
for hyperbolic spin models, but the ones
in~\eqref{e:SSBH2}--\eqref{e:SSBH02} capture the relevant phenomena
from the perspective of the Anderson
transition~\cite[Section~4.2]{MR2728731} as well as from the
perspective of the associated probabilistic models, as will be
discussed in Sections~\ref{sec:walks} and~\ref{sec:forests}.

We briefly summarise when SSB occurs for the
$\mathbb{H}^{2}$, $\mathbb{H}^{0|2}$, and $\mathbb{H}^{2|2}$ models.
In $d=2$, \eqref{e:SSBH2} and \eqref{e:SSBH02} do not hold for any $\beta>0$.
These results are versions of the Mermin--Wagner theorem
\cite{MR4021254,MR4218029,1911.08579,MR2561431}.
The situation is different in $d\geq 3$.
For the $\mathbb{H}^2$ model, SSB occurs for any $\beta>0$ as a result of convexity~\cite{MR3469136}.
The $\mathbb{H}^{2|2}$ and $\mathbb{H}^{0|2}$ models, however, have phase transitions:
SSB in the form of \eqref{e:SSBH2} and \eqref{e:SSBH02}, respectively,
occurs in $d\geq 3$ if and only if $\beta$ is sufficiently large~\cite{MR2728731,2107.01878}.
Once SSB is known to occur (or not), 
it is interesting and physically relevant to ask more precise
questions, e.g., about 
the asymptotics of the correlation functions $\langle
x_ix_j\rangle_{\beta,h}$. 
Sections~\ref{sec:walks} and~\ref{sec:forests} will discuss SSB
and sharper questions for the $\mathbb{H}^{2|2}$ and
$\mathbb{H}^{0|2}$ models.

\paragraph{Horospherical coordinates}

Important tools the study of the above models are horospherical coordinates
for the superspaces $\mathbb{H}^{n|2m}$ with $n\geq 2$ \cite{MR2736958,MR2728731}.
For the hyperbolic plane $\mathbb{H}^2$ these are coordinates $(t,s) \in \mathbb{R}^2$ such that
\begin{equation}
  x=\sinh(t)-\frac12 e^t|s|^2, \qquad y=e^ts, \qquad z=\cosh(t)+\frac12 e^t |s|^2.
\end{equation}
For the space $\mathbb{H}^n$ these coordinates generalise by taking
$s= (s^{i}) \in \mathbb{R}^{n-1}$.  For the
superspaces $\mathbb{H}^{n|2m}$ in addition there are $m$ pairs
Grassmann coordinates
$\psi = (\psi^{i}),\bar\psi =(\bar\psi^{i})$ such that
\begin{equation} 
  x=\sinh(t)-\frac12 e^t|s|^2- e^t\psi\bar\psi, \;\;
  y=e^ts, \;\;
  \xi = e^t\psi,\;\;
  \eta=e^t\bar \psi, \;\;
  z=\cosh(t)+\frac12 e^t |s|^2+e^t\psi\bar\psi,
\end{equation}
where we are using the abbreviation
$\psi\bar\psi = \sum_{i=1}^m \psi^i\bar\psi^i$ if there are $m$
Grassmann components.  In these coordinates the action becomes
\begin{align} \label{e:Hhoro}
  H_{\beta,h} &= \beta \sum_{ij} \left( (\cosh(t_i-t_j) -1) + \frac{1}{2} e^{t_i+t_j} |s_i-s_j|^2 + e^{t_i+t_j} (\psi_i-\psi_j)(\bar\psi_i-\bar\psi_j) \right)
  \nonumber\\
  &\;\;+ h \sum_i \left( (\cosh(t_i)-1)+\frac12 e^{t_i} |s_i|^2 +  e^{t_i} \psi_i\bar\psi_i \right),
\end{align}
and the hyperbolic reference measure is
$dt \, ds \, \partial_{\bar \psi} \, \partial_{\psi}\,
e^{(n-2m-1)\sum_i t_i}$, where $\partial_{\bar \psi} \partial_{\psi}$
denotes Grassmann integration if $m>0$.  A crucial feature
of~\eqref{e:Hhoro} is that the $s$ and $\psi,\bar\psi$ variables
appear quadratically in $H_{\beta,h}$ and hence can be integrated out
via exact Gaussian computations.
The $t$-marginal is thus proportional to
the \emph{positive} measure $e^{-\tilde H_{\beta,h}} \, dt$ where
\begin{align} \label{e:tildeH}
  \tilde H_{\beta,h}(t)
  &=
    \beta \sum_{ij} (\cosh(t_i-t_j)-1) + h \sum_i (\cosh(t_i)-1)
    \nonumber\\
  &\;\;
      + \frac{n-2m-1}{2} (\log \det (-\Delta_{\beta(t)}+h(t))-2\sum_i t_i).
\end{align}
In~\eqref{e:tildeH} $-\Delta_{\beta(t)}+h(t)$ is the $t$-dependent matrix acting as
\begin{equation} \label{e:Deltabetat}
  (-\Delta_{\beta(t)}f+h(t) f)_i = -\sum_{j\sim i} \beta e^{t_i+t_j}
  (f_j-f_i) + h e^{t_i} f_{i},
\end{equation}
with the $t$-dependent weights $\beta_{ij}(t)=\beta e^{t_{i}+t_{j}}$
and $h_{i}(t)=he^{t_{i}}$; this generalises immediately to edge- and
vertex-dependent weights.  The determinant in~\eqref{e:tildeH} arises
from the Gaussian integration over the $s$ and $\psi,\bar\psi$
variables.  Since the $t$-field is distributed according to a positive
measure, one can use standard tools from analysis.  This is useful
since, e.g., for all $\mathbb{H}^{n|2m}$ models,
\begin{equation}
  \langle z_i \rangle_{\beta,h} =
  \langle z_i+x_i \rangle_{\beta,h} =
  \langle e^{t_i} \rangle_{\beta,h},
\end{equation}
where the first identity used that $\langle x_i \rangle_{\beta,h} =0$, by symmetry.
For the pinned expectations, analogous
representations hold with $h=0$ and $t_0=0$.

\section{Physical background: Anderson transition}
\label{sec:random-band-matrices}

This section briefly discusses the origins of hyperbolic spin systems
as simplified models for the Anderson delocalisation--localisation
transition.  For a more detailed survey about this, we refer in
particular to \cite{MR3204347}.  Further excellent surveys
include~\cite{MR2953867} and \cite{MR708812,MR1843511} for a physics
perspective.  For general background on the Anderson transition, see
\cite{MR3364516}.

Consider a random matrix $H = (H(i,j))_{i,j\in\Lambda}$ such as the
Anderson Hamiltonian $H= H_\beta = -\beta\Delta + V$ where
$V=(V_i)_{i\in\Lambda}$ is an i.i.d.\ Gaussian potential, 
$\Lambda$ is a discrete torus approximating $\mathbb{Z}^d$,
and $\Delta$ is the lattice Laplacian on $\Lambda$.  The
fundamental question is to determine whether or not the spectrum of
$H$ contains an absolutely continuous part in the infinite volume
limit, and very closely related to this, if the eigenfunctions 
(often called states in this context) of $H$ are extended or localised.
Extended states correspond to a metallic phase while localised states
correspond to an insulating phase.  To discuss this
further, define the two-point correlation function 
\begin{equation}
  \label{eq:2pt}
  \tau_{\beta,E,h}(j,k) = \mathbb{E} |(H_\beta-E-ih)^{-1}(j,k)|^2, \quad
  j,k\in \Lambda,
\end{equation}
where $i=\sqrt{-1}$.  The existence of extended states for energies
near $E$ is essentially implied by
$\lim_{h\downarrow 0}\lim_{\Lambda\to \mathbb{Z}^d}
\tau_{\beta,E,h}(j,j) < \infty$.  For the Anderson model it is a
long-standing conjecture that this occurs in $d \geq 3$ for $E$ inside
the spectrum of $-\beta\Delta$ when $\beta$ is sufficiently large.  In
the same setting, the more precise quantum diffusion conjecture
asserts
\begin{equation}
  \lim_{h\downarrow 0} \lim_{\Lambda\to \mathbb{Z}^d} \tau_{\beta,E,h}(j,k) \approx
  D(E,\beta)(-\Delta)^{-1}(j,k) \sim C(E,\beta)|j-k|^{-(d-2)}, \quad
  j,k\in \mathbb{Z}^{d}
\end{equation}
for some constants $C,D$, and where the asymptotics hold for
$|j-k|\to\infty$.  This gives a hint that the conjecture might be
difficult: the two-point function decays slowly, like that of the
massless Gaussian free field.  Such behaviour also occurs for
fluctuations of spontaneously broken continuous symmetries (Goldstone
modes). In \cite{10.1007/BF01319839,MR708812} it was argued that
the origin of
extended states is SSB of a (complicated) spin model with hyperbolic symmetry,
and that quantum diffusion is exactly the associated Goldstone mode.
The spin model 
is based on the supersymmetric approach to the replica trick for
computing the two-point function.

We briefly indicate some parallels between the present discussion and
Section~\ref{sec:hyp}.  The elementary identity
\begin{equation}
  \frac{1}{h} \mathbb{E} \operatorname{Im} (H_\beta-E-ih)^{-1}(j,j) = \sum_{k} \mathbb{E} |(H_\beta-E-ih)^{-1}(j,k)|^2,
\end{equation}
which is also valid without expectations, is analogous to the Ward
identities \eqref{e:ward1}--\eqref{e:ward2}.  Thus the role of
$\langle z_j \rangle$ is played by
$\mathbb{E} \operatorname{Im} (H_\beta-E-ih)^{-1}(j,j)$. In the limit
$h\downarrow 0$ this is $\pi$ times the density of states $\rho(E)$,
i.e., the asymptotic eigenvalue distribution. The role of the
two-point functions $\langle x_jx_k \rangle$ or
$\langle \xi_j\eta_k\rangle$ is played by
$\tau_{\beta,E,h}(j,k)=\mathbb{E} |(H_\beta-E-ih)^{-1}(j,k)|^2$.  The
absolute values in the latter correlation function are essential and
the origin of the hyperbolic symmetry~\cite[Section~2.3]{MR3204347}.
The noncompactness of the hyperbolic symmetry manifests itself in the
high temperature phase: the unboundedness of $\tau_{\beta,E,h}(j,k)$
as $h\downarrow 0$ signals an absence of delocalisation. The stronger
notion of localisation corresponds to
\begin{equation}
  \tau_{\beta,E,h}(j,k) \approx \frac{e^{-c|j-k|}}{h}.
\end{equation}
The divergence as $h\downarrow 0$ is analogous to the behaviour of the
$\mathbb{H}^{2|2}$, see Section~\ref{sec:hyp}, and is different from
that of spin systems with compact symmetry.  For further discussion,
see \cite{MR3204347}.

\paragraph{Dictionary}
\label{sec:dictionary}
The analogies between the expected behaviours of the Anderson model and
the probabilistic models described in the next two sections is
summarised below. In $d=2$, $\beta_{c}=\infty$, while
$\beta_{c}<\infty$ for $d\geq 3$.
\begin{center}
\begin{tabular}{c c c} 
  \toprule
  & $\beta<\beta_{c}$ & $\beta>\beta_{c}$ \\
  \midrule
  Anderson Model & localised (insulating) phase & extended (metallic) phase \\
  VRJP & positive recurrent phase & transient phase \\
  Arboreal gas & subcritical percolation phase & percolating phase \\
 \bottomrule
\end{tabular}
\end{center}
Logically there is the possibility of non-extended states that are not
localised, which would correspond to a null-recurrent phase for the
VRJP and a phase of the arboreal gas where infinite clusters do not
occur, but the cluster size distribution has infinite mean.

\section{Linearly reinforced walks and $\mathbb{H}^{2|2}$}
\label{sec:walks}

Formulas arising in the study of the $\mathbb{H}^{2|2}$ model (e.g.,
\eqref{e:Deltabetat}) have interpretations in terms of random walks,
and similarities with ERRW did not go
unnoticed~\cite[Section~1.5]{MR2728731}. This was given an explanation
by Sabot and Tarr\`{e}s~\cite{MR3420510}; the explanation passes
through another reinforced random walk, which we now introduce.  Fix
edge weights $\beta_{ij}> 0$ for each edge $ij\in E$, and set
$\beta_{ij}=0$ if $ij \not\in E$.  The \emph{vertex-reinforced jump
  process (VRJP)} with $X_{0}=0$ is the continuous-time
self-interacting random walk with transition probabilities
\begin{equation}
  \label{eq:VRJP}
  \mathbb{P}^{\text{VJRP}(\beta)}_{0}[X_{t+dt}=j | (X_{s})_{s\leq t},X_{t}=i] = \beta_{ij}
  L^{j}_{t}, \quad L^{j}_{t} = 1+ \int_{0}^{t}1_{X_{s}=j}\,ds.
\end{equation}
The quantity $L^{j}_{t}$ is the \emph{local time} at $j$ at time $t$,
up to the shift by $1$.  In words, then, conditionally on the shifted
local times at time $t$ and that $X_{t}=i$, a VRJP jumps to site $j$
with probability proportional to $\beta_{ij}L^{j}_{t}$.  Thus previously vertices
visited are preferred.  The amount of local time accrued at $i$
before jumping away has the distribution of an exponential random variable with rate
$\sum_{j}\beta_{ij}L^{j}_{t}$. With this in mind, large edge weights
$\beta_{ij}$ heuristically correspond to weak reinforcement: jumps
occur quickly and do not alter the local time profile too much.

Sabot and Tarr\`{e}s gave an exact formula for the (properly scaled)
limiting local times of the VRJP, and explained that this distribution
is also the distribution of the $t$-field of the $\mathbb{H}^{2|2}$
model. They further showed that the magic formula for the ERRW follows
from this result, see Section~\ref{sec:further-walks} below.
Similarly to the ERRW, the VRJP can be expressed as a continuous-time
random walk in a random environment. The next theorem is a slightly
informal statement of this result. The precise formulation requires
looking at the VRJP in the correct time parameterisation;
see~\cite{MR3420510}. For a symmetric square matrix $A$ with
$\sum_j A_{ij}=0$ for all rows $i$, we write ${\det}^0(A)$ for the
value of any principal cofactor of $A$, e.g., the determinant with the
first row and column of $A$ removed.

\begin{theorem}[Magic formula for VJRP~\cite{MR3420510}]
  \label{thm:VRJP-magic}
  Let $G=(\Lambda,E)$ be a finite graph with $|\Lambda|=N$.
  In the exchangeable time parameterisation of the VRJP,
  \begin{equation}   \label{e:VRJP-magic}
    \mathbb{P}^{\text{VJRP}(\beta)}_{0}[\cdot] =
    (2\pi)^{-\frac{N-1}{2}}
    \int_{\mathbb{R}^{\Lambda\setminus 0}} \mathbb{P}^{\text{SRW}(c(t))}_0[\cdot] \, e^{-\frac\beta 2 \sum_{i,j} \cosh(t_i-t_j)} ({\det}^0(-\Delta_{\beta(t)}))^{\frac12} \prod_{k\in\Lambda\setminus0}\, e^{-t_k} \, dt_k.
  \end{equation}
  where $\mathbb{P}^{\text{SRW}(c(t))}_0$ is the distribution of a
  continuous-time simple random walk with conductances
  $c(t)_{ij} = \beta e^{t_i+t_j}$ started at $0$.
\end{theorem}

The measure on the right-hand side of \eqref{e:VRJP-magic} is exactly the horospherical $t$-marginal of the $\mathbb{H}^{2|2}$ model
(with $h=0$ and pinned at $0$).
The existence of a phase transition between a transient and a recurrent phase of the VRJP on
$\mathbb{Z}^{d}$ for $d\geq 3$ now essentially follows from the following earlier
{results for the $\mathbb{H}^{2|2}$ model (and extensions to the pinned model):

\begin{theorem}[SSB for $\mathbb{H}^{2|2}$ \cite{MR2728731}]
  Let $d \geq 3$ and  $\beta \geq \beta_1$. There exists
  $C_\beta>0$ such that
  \begin{equation}
    \lim_{h\downarrow 0}\lim_{\Lambda \to \mathbb{Z}^d}\langle \cosh(t_i)^8 \rangle_{\beta,h} \leq C_\beta.
  \end{equation}
  Similar statements hold for other observables and for the pinned model.
\end{theorem}

\begin{theorem}[Localisation for $\mathbb{H}^{2|2}$ \cite{MR2736958}]
  \label{thm:h22loc}
  Let $d \geq 1$ and $\beta \leq \beta_0$. There exist
  $C_\beta,c_\beta>0$ such that
  \begin{equation}
     \langle x_ix_j \rangle_{\beta,h} \leq \frac{C_\beta}{h} e^{-c_\beta |i-j|}.
  \end{equation}
  Similar statements hold for other observables and for the pinned model.
\end{theorem}

The existence of a recurrent phase for small $\beta$ has also been proved
more directly from the definition of the VRJP \cite{MR3189433}. 
A proof of transience
that only uses the random walk point of view seems challenging, and
would be of interest.

\subsection{Hyperbolic symmetry and the VRJP}
\label{sec:hyperb-symm-vrjp}

A more direct and general connection between hyperbolic spin systems
and the VRJP was found later~\cite{MR4021254}. Towards this, observe
(as was already done in~\cite{MR3420510}) that the joint process
$(X_{t},L_{t})$ of the VRJP and its local time is a Markov process,
where $L_{t}=(L^{j}_{t})_{j\in V}$. The infinitesimal generator
$\mathcal{L}$ of the joint process acts on
$g\colon V\times [0,\infty)^{V}\to \mathbb{R}$ by
\begin{equation}
  \label{eq:VRJP-gen}
  \mathcal{L} g(i,\ell) = \sum_{j}\beta_{ij}\ell_{j}(g(j,\ell)-g(i,\ell)) +
  \frac{\partial g(i,\ell)}{\partial \ell_{i}}.
\end{equation}
Write $\mathbb{E}^{\text{VRJP}(\beta,\ell)}_{i}$ for the expectation of the
joint process with initial vertex $i$ and local times
$\ell = (\ell_{i})_{i\in\Lambda}$. The definition~\eqref{eq:VRJP}
corresponds to $\ell_{i}=1$ for all $i$.

To connect the VRJP to hyperbolic symmetry, consider the
$\mathbb{H}^{2}$ model, for example, and recall the infinitesimal
generator $T_i$ of Lorentz boosts in the $x_{i}z_{i}$-plane acting at vertex
$i$ from \eqref{eq:Lorentz}.  Then under mild
hypotheses on $G$, integration by parts and~\eqref{eq:Lorentz} yields
\begin{equation}
  \label{eq:Iso-1}
  -\sum_{j}\int_{(\mathbb{H}^{2})^{\Lambda}} (\mathcal{L}G(j,z))
  x_{i}x_{j} e^{-H_{\beta,0}(u)} \prod_{k\in\Lambda}du_{k}
  = \int_{(\mathbb{H}^{2})^{\Lambda}} 
  (T_{i}x_{i}) G(j,z)  e^{-H_{\beta,0}(u)}\prod_{k\in\Lambda}du_{k}.
\end{equation}
Thus boosts are adjoint to the generator of the VJRP.
A consequence is the next theorem. 
\begin{theorem}
  \label{thm:BHS-iso}
  Consider the $\mathbb{H}^{2}$ model. If
  $F\colon \mathbb{R}^{\Lambda}\to\mathbb{R}$ decays fast enough, then
  \begin{equation}
    \label{eq:BHS-iso}
    \langle x_{i}x_{j}F(z)\rangle_{\beta,0} = \langle
    z_{i}\int_{0}^{\infty}dt\, \mathbb{E}^{\text{VRJP}(\beta,z)}_{i}[F(
    L_{t})1_{X_{t}=j}]\rangle_{\beta}. 
  \end{equation}
\end{theorem}
\begin{proof}[Sketch of proof]
  Normalise~\eqref{eq:Iso-1} and choose $G(j,\ell) = G_{t}(j,\ell) =
  \mathbb{E}^{\text{VRJP}(\beta,\ell)}_{j}F(L_{t})$. Since
  $(X_{t},L_{t})$ is a Markov process with generator $\mathcal{L}$,
  we have $\mathcal{L} G_{t}=\partial_{t}G_{t}$.
  Integrating the resulting identity over $t$ in $(0,\infty)$
  gives the result.
\end{proof}

Theorem~\ref{thm:BHS-iso} shows that $\mathbb{H}^{2}$ quantities can
be computed in terms of the averages of VRJP quantities, the average
being over the initial local time of the VRJP. This average is
inconvenient for studying the VRJP itself. The computations above,
however, immediately generalise to other hyperbolic spin models. For
the $\mathbb{H}^{2|2}$ model one can in addition use Theorem~\ref{thm:susyloc} to
exactly compute the undesirable average. The result is the following theorem.

\begin{theorem}
  \label{thm:BHS-iso-susy}
  Consider the $\mathbb{H}^{2|2}$ model. Then for any $F\colon \mathbb{R}^\Lambda \to \mathbb{R}$ that decays fast enough,
  \begin{equation}
    \label{eq:BHS-iso}
    \langle x_{i}x_{j}F(z)\rangle_{\beta,0} = \int_{0}^{\infty}dt\,
    \mathbb{E}^{\text{VRJP}(\beta)}_{i} [ F(L_{t})1_{X_{t}=j}].
  \end{equation}
\end{theorem}
In particular, $\langle x_{i}^{2} \rangle_{\beta,h}$ is the expected time the
VRJP started from $i$ spends at $i$ when killed at rate $h>0$.
This relation can be used to
prove the VRJP is recurrent in two dimensions, irrespective of the
reinforcement strength $\beta>0$, by proving a Mermin--Wagner theorem
for the $\mathbb{H}^{2|2}$ model~\cite{MR4021254}. Informally,
Mermin--Wagner theorems assert that continuous symmetries cannot be
spontaneously broken in $d=1,2$. As discussed earlier, for the
$\mathbb{H}^{2|2}$ model SSB corresponds to a finite variance, i.e.,
transience.

\paragraph{Isomorphism theorems} Theorems~\ref{thm:BHS-iso}
and~\ref{thm:BHS-iso-susy} are examples of \emph{isomorphism
  theorems}, meaning identities relating the local time field of a
stochastic process to a spin system. The first example of such a
result related simple random walk to the Gaussian free field and was
obtained by Brydges, Fr\"{o}hlich, and Spencer~\cite{MR648362}. They
were inspired by Symanzik~\cite{Syma69}. The formulation as a
distributional identity is due to Dynkin~\cite{MR693227}; sometimes
the result is called the BFS-Dynkin isomorphism. A host of other
isomorphism theorems have been found in Gaussian settings,
see~\cite{MR2250510}. Other isomorphism theorems for the VRJP can be
obtained by the approach above, and it is possible to obtain
Theorem~\ref{thm:VRJP-magic} in this
way. See~\cite{MR4255180}. Isomorphisms for the VRJP can also be
obtained by expressing the VRJP as a mixture of Markov processes and
using isomorphism theorems for the Markov processes,
see~\cite{1911.09036}.

\subsection{Random Schr\"odinger representation and STZ field}
\label{sec:STZ}

In \cite{MR2736958}, it was observed that after conjugation by the diagonal matrix $e^{-t} = (e^{-t_i})_i$,
the matrix $-\Delta_{\beta(t)}+h(t)$ in \eqref{e:Deltabetat} becomes a Schr\"odinger operator with $t$-dependent potential:
\begin{equation}
  \label{e:V}
  e^{-t} \circ (
  -\Delta_{\beta(t)} + h(t)) \circ e^{-t}
  = -\Delta_\beta + V(t), \qquad V_i(t) = \sum_{j} \beta_{ij}(e^{t_j-t_i}-1) + h_i e^{-t_i}.
\end{equation}
This point of view led to the proof of Theorem~\ref{thm:h22loc}.  It
was later recognised that this random Schr\"odinger point of view can
be used to obtain a powerful representation of the
$t$-field~\cite{MR3729620}.  For the pinned $\mathbb{H}^{2|2}$ model
with $h=0$ and $t_0=0$, the $t$-field measure \eqref{e:tildeH} can be
written in terms of $-\Delta_\beta+V(t)$ using that
\begin{equation}
  e^{-\tilde H_\beta(t)}
  =
  e^{-\frac{1}{2}\sum_i V_i(t)}
  (\det(-\Delta_\beta +V(t)))^{1/2}
  .
\end{equation}
This suggests it might be useful to change variables from $t$ to
$V(t)$. Implementing this change of variables requires
taking into account that when $t_0=0$, the map $t\mapsto V(t)$ is not surjective onto the
set of $V$ such that $(-\Delta_\beta + V)$ is positive
definite. This can be sidestepped by treating $V$ as the fundamental
variable, i.e., considering
\begin{equation}
  \label{e:V-field}
  e^{-\frac12 \sum_i V_i}
  (\det(-\Delta_\beta+V))^{-1/2}
  \, \mathbf{1}(-\Delta_\beta + V \text{ is positive definite}) \, dV.
\end{equation}
The random vector $B_i = \frac12(V_i+\sum_j \beta_{ij})$ is
often called the `$\beta$-field,'
but since we use $\beta$ for edge weights (inverse temperature), we will
denote it by $B$ instead 
and call it the \emph{STZ field}.

\begin{theorem}
  \label{thm:stz}
  The Laplace transform of the STZ field is given by
  \begin{equation}
    \mathbb{E} e^{-(\lambda,B)} =
    \prod_i \frac{1}{(\lambda_i+1)^{1/2}}
    \prod_{ij} e^{-\beta_{ij}(\sqrt{\lambda_i+1}\sqrt{\lambda_j+1}-1)}
    .
  \end{equation}
 Moreover, the $t$-field (pinned at any vertex) can be recovered in distribution from $B$.
\end{theorem}
In particular, the theorem implies the STZ field is $1$-dependent for $\mathbb{H}^{2|2}$.
In \cite{MR3904155}, this remarkable property of the STZ field was
used to construct an infinite volume version on $\mathbb{Z}^d$, and
applied to characterise transience and recurrence of the VRJP in terms
of a $0/1$ law.

\subsection{Phase diagram of the VRJP}
\label{sec:phase-diagram-vrjp}

The most basic qualitative question one can ask about the VRJP is
whether it is recurrent or transient for a given reinforcement
strength $\beta>0$. This may in principle depend on the
precise notion of recurrence used, as the VRJP is non-Markovian. As
discussed above, for $d\geq 3$ the existence of a phase in which the
VRJP is almost surely recurrent was established
in~\cite{MR3420510,MR3189433}, and an almost surely transient phase
in~\cite{MR3420510}. For $d=2$, recurrence for all $\beta>0$ in the
sense of infinite expected local time at the initial vertex was
established in~\cite{MR4021254}. 
Proofs of almost sure recurrence followed
shortly~\cite{1911.08579,MR4218029}.

The qualitative behaviour of the VRJP is almost completely understood
on $\mathbb{Z}^{d}$ due to the following remarkable correlation
inequality of Poudevigne.

\begin{theorem} \label{thm:poudevigne} For the $\mathbb{H}^{2|2}$
  model and any convex function $f$, the expectation
  $\langle f(e^{t_j})\rangle_\beta^0$ is increasing in all weights
  $\beta=(\beta_{ij})$.
\end{theorem}
The proof of Theorem~\ref{thm:poudevigne} 
relies on the STZ field~\cite{1911.02181}. This inequality implies that transience is a monotone property with
respect to the constant initial reinforcement parameter $\beta$. Combined with the
results of the previous paragraph, this implies that the VRJP has a
sharp transition from almost sure recurrence to almost sure transience
on $\mathbb{Z}^{d}$ for constant 
$\beta$: recurrence for $\beta<\beta_{c}(d)$ and transience for $\beta>\beta_{c}(d)$.
The behaviour at $\beta_{c}$ is  open.
Poudevigne's correlation inequality also leads to a proof of recurrence in $d=2$.

\subsection{Further discussion}
\label{sec:further-walks}

\paragraph{Back to edge-reinforced random walk}
\label{sec:edge-reinf-rand}

The connection of the ERRW to the $\mathbb{H}^{2|2}$ model is somewhat
less direct than for the VRJP: it turns out that the ERRW is an
average of VRJPs \cite{MR3420510}. Somewhat more precisely, ERRW with
initial edge weights $\alpha$ can be obtained from the VRJP with
initial edges weights $\beta$ if the $\beta_{ij}$ are chosen to be
independent Gamma random variables with mean $\alpha_{ij}$. While this
additional randomness presents some difficulties, the existence of a
transient phase for the ERRW in $d\geq 3$ was obtained by similar
methods to that of the VRJP~\cite{MR3366053}. In terms of the spin
model, the Gamma distributed random edge weights correspond to
replacing the exponential
$e^{\sum_{ij} \beta_{ij}(u_i \cdot u_j+1)}$ by
$\prod_{ij} (-u_i\cdot u_j)^{-\alpha_{ij}}$ in the (super)
measure. Such a product weight is often called a Nienhuis interaction.

Interestingly, the recurrence of the ERRW in two dimensions was
obtained before the recurrence of the VRJP. This was possible due to
insights of Merkl and Rolles, who directly proved a Mermin--Wagner
type theorem for the ERRW by making use of the magic
formula~\cite{MR2561431}. Merkl and Rolles were able to conclude recurrence of
the ERRW on $\mathbb{Z}^{2}$ for strong reinforcement if each edge of
the lattice was replaced by a long path. Sabot and Zeng's proved
recurrence on $\mathbb{Z}^{2}$ for all reinforcement strengths
by obtaining a characterisation of recurrence in terms of the STZ
field~\cite{MR3904155}, and showing that an estimate from~\cite{MR2561431} implies
recurrence. The ergodic properties of the STZ field play a crucial
role in this argument.

\paragraph{Beyond $\mathbb{Z}^{d}$} There are also results for the
VRJP beyond $\mathbb{Z}^{d}$. The existence of a transition on trees
was proven in~\cite{MR2027294}, and on non-amenable graphs
in~\cite{MR3189433}. A fairly complete understanding on trees has been
obtained, see~\cite{MR2985176} and references therein. 

\paragraph{Future directions}
There remain many open questions. What is the critical behaviour of
the VRJP and the $\mathbb{H}^{2|2}$ model on $\mathbb{Z}^{d}$, $d\geq
3$? Is there an upper critical dimension?
For $d=3$ aspects of this question were studied numerically in
\cite{PhysRevB.54.12763}, and evidence was found for
the existence of a multifractal structure in the $\mathbb{H}^{2|2}$
model. Multifractal structure
is also expected near the Anderson transition for random Schr\"odinger operators.
For the regular tree (Bethe lattice), further remarkable critical behaviour
  was observed, in part numerically, in \cite{MR863830}. 
  This reference
concerns a more complicated sigma model,
but the main predictions also apply to the $\mathbb{H}^{2|2}$ model \cite{PhysRevB.54.12763}. 
On $\mathbb{Z}^{2}$ the VRJP is believed to be
\emph{positive} recurrent, i.e., exponentially localised, but this
important conjecture about the $\mathbb{H}^{2|2}$ model and the VRJP
remains open. The heuristic for positive recurrence 
is based on the (marginal) renormalisation group flow
and goes along with the prediction of asymptotic freedom at short distances~\cite[Section~4.3]{MR2728731}. Analogous predictions 
based on analogous heuristics exist for the 2d Heisenberg model, the 2d Anderson model, 4d non-abelian Yang-Mills theory,
and the 2d arboreal gas (discussed below).
Another question is
to understand the VRJP in $d\geq 3$ with non-constant initial local
times: Theorem~\ref{thm:BHS-iso} and results
of~\cite{MR2104878} suggest the VRJP is always transient if started with
initial local times given by the $z$-field of the $\mathbb{H}^{2}$
model. Understanding the properties of the $z$-field that destroy the
phase transition would be interesting.

\section{The arboreal gas and $\mathbb{H}^{0|2}$}
\label{sec:forests}

The arboreal gas is the uniform measure on (unrooted spanning) forests
of a weighted graph.  More precisely, given an undirected graph
$G=(\Lambda,E)$, a forest $F=(\Lambda,E(F))$ is an acyclic subgraph of
$G$ having the same vertex set as $G$.  Given an edge weight $\beta>0$
(inverse temperature) and a vertex weight $h\geq 0$ (external field),
the probability of an edge set $F$ under the arboreal gas measure is
\begin{equation} \label{e:P-forest}
  \mathbb{P}_{\beta,h}
  [F] = \frac{1}{Z_{\beta,h}}
  \beta^{|E(F)|} \prod_{T\in F} (1+h|V(T)|)
  \,
  \mathbf{1}(\text{$F$ is a forest})
\end{equation}
where $T\in F$ denotes that $T$ is a tree in the forest $F$, i.e., a
connected component of $F$.
We write $\mathbb{P}_\beta = \mathbb{P}_{\beta,0}$.  As for the VRJP, the
generalisation to edge- and vertex-dependent weights
$\beta=(\beta_{ij})$ and $h=(h_i)$ is straightforward, and is sometimes
useful.

The arboreal gas arises naturally in the context of the $q$-state
random cluster model ($q$-RCM), which we recall is the model defined
by~\eqref{e:P-forest} by omitting the indicator function and instead
weighting each component by a factor $q>0$. In particular, $q=1$ is
Bernoulli bond percolation. On a finite graph, the arboreal gas edge
weight $\beta'$ is the limit of the $q$-RCM as $q,\beta\to 0$ such
that $\beta/q\to\beta'$, and it is natural to think of the arboreal
gas as the $0$-RCM.  The most fundamental question about the arboreal
gas is whether or not it has a percolation phase transition.  It is
straightforward to establish a subcritical phase when $\beta$ is
small: the arboreal gas can be stochastically dominated by bond
percolation~\cite[Theorem~3.21]{MR2243761}, and for $\beta$ small the
domination is by subcritical percolation.

\subsection{Phase transitions for the arboreal gas}
\label{sec:arbphase}

The existence of a supercritical phase for the
arboreal gas is a more subtle question than
for the $q$-RCM with $q>0$. One way to see this subtlety is based on
symmetries. To discuss this, recall that
for $q\in\{2,3,\dots\}$ there is a 
connection between the $q$-RCM and the $q$-state Potts
model~\cite{MR359655}. In particular, spin-spin correlations in the
$q$-state Potts model are equivalent to connection probabilities in
the $q$-RCM. The results of \cite{MR2110547,MR2142209} extend
  this relationship to $q=0$: the $\mathbb{H}^{0|2}$ model is a
spin representation of the arboreal gas:
\begin{theorem}
  \label{thm:h02arb}
  Let $\langle\cdot\rangle_{\beta}$ and $\mathbb{P}_{\beta}$ denote
  the $\mathbb{H}^{0|2}$ and arboreal gas measures on a finite
  graph. For vertices $i,j \in \Lambda$,
  \begin{equation}
    \label{eq:h02arb}
    \mathbb{P}_{\beta}[i\leftrightarrow j] = -\langle u_{i}\cdot u_{j}\rangle_{\beta}.
  \end{equation}
  Moreover, the partition functions of the $\mathbb{H}^{0|2}$ model
  and the arboreal gas coincide.
\end{theorem}
Strictly speaking, the $\mathbb{H}^{0|2}$ formulation of
Theorem~\ref{thm:h02arb} first occurred in~\cite{MR4218682} as a
reformulation of~\cite{MR2110547,MR2142209}; the hyperbolic point of view
plays an important role in the proof of
Theorem~\ref{thm:arboreal-magic} below.
Theorem~\ref{thm:h02arb} suggests the existence of a supercritical
phase for the arboreal gas may depend on the dimension, as strong
connection probabilities corresponds to a symmetry breaking phase
transition for the $\mathbb{H}^{0|2}$ model. Unlike the $q$-Potts
models with $q\in\{2,3,\dots\}$, this model possesses a continuous
symmetry, so one might expect a Mermin--Wagner theorem to prevent such
a transition in $d=2$. This is indeed true:

\begin{theorem}[{\cite[Theorem~1.3]{MR4218682}}] 
  \label{thm:arbd2}
  Let $d=2$.
  For any $\beta> 0$, there exists $c(\beta)>0$
  such that $\mathbb{P}_\beta^{\Lambda}[0\leftrightarrow j] \leq
  |j|^{-c(\beta)}$
  for any $\Lambda \subset \mathbb{Z}^2$.
\end{theorem}

It is possible to predict Theorem~\ref{thm:arbd2} without knowing
about the $\mathbb{H}^{0|2}$ spin representation as
follows~\cite{PhysRevLett.98.030602}. The critical value of $\beta$
for the $q$-RCM on $\mathbb{Z}^{2}$ with $q\geq 1$ is known to be
$\beta_{c}(q)= \sqrt{q}$~\cite{MR2948685}, and this self-dual point is
predicted to be the critical point for all $q>0$.  Since the arboreal
gas $\mathbb{P}_{\beta'}$ is the limit of the $q$-RCM with
$\beta=\beta' q$, if the location of the critical point is continuous
as $q\downarrow 0$, it follows that $\beta_{c}(0)=\infty$.  These
heuristics support the conjecture that connection probabilities of the
$2d$ arboreal gas decay exponentially for any $\beta>0$.  Independent
support for this conjecture can be obtained by renormalisation group
heuristics, almost exactly as for the $2d$ VRJP~\cite{MR2110547}.

Turning the preceding paragraph into a rigorous proof would be very
interesting.  It would also be interesting to have a probabilistic
proof (in terms of forests) that the arboreal gas does not have a
phase transition on $\mathbb{Z}^2$.  The proof of
Theorem~\ref{thm:arbd2} follows different lines. A key step is the
following, which reduces the proof to an adaptation
of~\cite[Theorem~1]{MR4218029}.
\begin{theorem}[Magic formula for arboreal gas]
  \label{thm:arboreal-magic}
  Let $G=(\Lambda,E)$ be a finite connected graph. For vertices $0,j\in \Lambda$,
  \begin{equation}
    \label{eq:arboreal-magic}
    \mathbb{P}_\beta[0\leftrightarrow j] =
    \frac{1}{Z_\beta}
    \int_{\mathbb{R}^{\Lambda\setminus 0}} e^{t_j} \, e^{-\frac\beta 2 \sum_{i,k} \cosh(t_i-t_k)} ({\det}^0(-\Delta_{\beta(t)}))^{3/2} \prod_{k\in\Lambda\setminus 0}\, e^{-3t_k} \, dt_k,
  \end{equation}
  where ${\det}^0(-\Delta_{\beta(t)})$ denotes any principal cofactor
  of $-\Delta_{\beta(t)}$.
\end{theorem}
In outline, the proof of Theorem~\ref{thm:arboreal-magic} consists of
three steps: Theorem~\ref{thm:h02arb}, rewriting the
$\mathbb{H}^{0|2}$ expectation in terms of $\mathbb{H}^{2|4}$ by SUSY
localisation, and then changing to horospherical coordinates and
integrating out all but the $t$-field. The magic formula for the VRJP
from Theorem~\ref{thm:VRJP-magic} has a strikingly similar form, but
with the two occurrences of $3$s in~\eqref{eq:arboreal-magic} replaced
by $1$s. This difference in powers is due to there being two
additional Grassmann Gaussian integrals for $\mathbb{H}^{2|4}$ as
compared to $\mathbb{H}^{2|2}$.

In three and more dimensions the arboreal gas does, however, undergo a
percolation phase transition.  To state a precise theorem, let
$\Lambda_{N} = \mathbb{Z}^{d}/L^{N}\mathbb{Z}^{d}$ denote a torus of
side-length $L^{N}$ with $L$ large. The next theorem immediately
implies that there is a macroscopic tree occupying most of the torus
with large probability.
\begin{theorem}[{\cite[Theorem~1.1]{2107.01878}}]
  \label{thm:arbd3+}
  Let $d\geq 3$. If $\beta$ is sufficiently large, then there exists
  $\theta_{d}(\beta)=1-O(1/\beta)$, $D(\beta)>0$, and $\kappa>0$ such that
  \begin{equation}
    \label{eq:arbd3+}
    \mathbb{P}^{\Lambda_N}_\beta[0\leftrightarrow j] = \theta_d(\beta)^2 +
    D(\beta)(-\Delta)^{-1}(0,j) +
    O\left(\frac{1}{\beta|j|^{d-2+\kappa}}\right).
  \end{equation}
  Similar asymptotics hold for other correlation functions.
\end{theorem}
The polynomial correction in Theorem~\ref{thm:arbd3+} is the hallmark
of critical behaviour in statistical mechanics, and is a manifestation
of the Goldstone mode associated with the broken continuous symmetry of the
$\mathbb{H}^{0|2}$ model at low temperatures. The proof of
Theorem~\ref{thm:arbd3+} relies essentially on the $\mathbb{H}^{0|2}$
representation (Theorem~\ref{thm:h02arb}), and is based on a
combination of Ward identities and a renormalisation group
analysis. The renormalisation group analysis is based in part on
methods developed previously in different contexts, in particular
\cite{MR2523458,MR3332938,MR3332939,MR3339164,MR3345374}.

\subsection{Further discussion}

In contrast to the VRJP, even the qualitative phase diagram of the
arboreal gas remains incomplete: we do not know the existence of a
$\beta_{c}$ such that percolation occurs if $\beta>\beta_{c}$, and
does not if $\beta<\beta_{c}$. It is also more difficult to discuss
the arboreal gas directly in the infinite volume limit than the VRJP;
the STZ field does not have finite dependence and is less obviously
useful, and useful correlation inequalities to this end remain
conjectural, see below.  Nonetheless, many open questions beckon.

\paragraph{Critical behaviour}
There is strong evidence that the upper critical dimension
of the arboreal gas is $d=6$, just as for bond percolation,
and that the critical behaviour is governed by more conventional
critical behaviour as compared to the $\mathbb{H}^{2|2}$ model~\cite{PhysRevLett.98.030602}.

\paragraph{Comparison with percolation}

Recall that the analogue of Theorem~\ref{thm:arbd3+} for Bernoulli
percolation has an \emph{exponentially} decaying
correction~\cite{MR1048927}. Informally this means that supercritical
percolation with the giant removed behaves like subcritical
percolation. This can be given a more precise meaning in the simpler
setting of the Erd\H{o}s--R\'{e}nyi random graph, i.e., Bernoulli
percolation on the complete graph $K_{N}$, where it is known as the
discrete duality principle~\cite[Section~10.5]{MR3524748}.

The polynomial correction in Theorem~\ref{thm:arbd3+} shows that the
arboreal gas does not satisfy a duality principle. Rather, its
supercritical phase behaves like a critical model off the giant. This
can again be given a more precise formulation on the complete graph
$K_{N}$ and on the wired regular tree, where detailed results are
known~\cite{MR1167294,MR3845513,2108.04335,2108.04287}.  In
particular, the exact cluster distribution can be determined: on $K_N$
in the supercritical phase there is a unique giant tree, and an
unbounded number of trees of size $\Theta(N^{2/3})$.

It is natural to predict the macroscopic behaviour of the arboreal gas
on $d$-dimensional tori $\Lambda_{N}$ with $d\geq 3$ is similar to
that on the complete graph. In particular, one expects a unique giant
tree. The next-order critical corrections can also be expected to be
similar, at least when $d\gg 3$. In particular, the second biggest
tree should have size comparable to $|\Lambda_N|^{2/3}$.  Similar
results have been established for critical Bernoulli percolation in
high dimensions, see~\cite[Chapter~13]{MR3729454}.  More ambitiously,
we expect the order statistics of the rescaled cluster size
distribution to be universal, i.e., the same as on the complete graph
as determined in~\cite{MR1167294,MR3845513}. This conjecture may be
easier to explore in other settings first, e.g., on expanders, where a
phase transition can be established by elementary
methods~\cite{MR3451159}.

\paragraph{Infinite-volume geometry and the UST}
There is a large body of literature in probability theory concerning
\emph{uniform spanning forests} (USF), meaning weak infinite-volume
limits of uniform spanning tree (UST) measures on finite graphs,
see~\cite[Chapter~10]{MR3616205}.  To avoid confusion with the
arboreal gas (sometimes also called the USF \cite{MR2724667}), we will
call these infinite-volume limits the UST on $\mathbb{Z}^d$. While the
component structure of the UST on a finite graph is not particularly
interesting, the infinite volume limit is: Pemantle proved that there
is a unique connected component on $\mathbb{Z}^{d}$ for $d\leq 4$, and
infinitely many connected components on $\mathbb{Z}^{d}$ for
$d>4$~\cite{MR1127715}. This happens as `long connections' can be lost
in the weak limit.

On a finite graph, the UST measure is the limit
$\beta\to\infty$ of the arboreal gas with edge weights $\beta$,
and it is natural to wonder if the arboreal gas at low temperatures
$\beta\gg 1$ has similar properties to the UST in infinite volume. For
global properties, this can evidently only happen when there is a
percolation transition.  We are therefore lead to ask: for $d\geq 3$
at low temperatures, is it the case that for $d=3,4$ the
infinite-volume arboreal gas has a unique infinite tree, while for
$d>4$ there are infinitely many infinite trees?  Is it the case that the
infinite components of the arboreal gas are topologically one-ended,
as for the UST?

\paragraph{Negative correlation}
A key tool in studying the $q$-RCM with $q\geq 1$ is that it is
\emph{positively associated}: for increasing functions
$f,g\colon \{0,1\}^{E}\to \mathbb{R}$, the covariance of $f$ and $g$
is non-negative. This is a special case of the FKG
inequality~\cite{MR0309498}. Positive association fails for $q<1$ and
for the arboreal gas. It is believed, but not known, that these cases
are in fact \emph{negatively associated}: for
$f,g\colon \{0,1\}^{E}\to\mathbb{R}$ depending on disjoint sets of
edges, the covariance of $f$ and $g$ is non-positive. Negative
association is more subtle than positive association, and the
development of flexible yet powerful theoretical frameworks is an
active subject~\cite{MR1757964,MR2476782,MR4172622,1811.01600}. While
some of this theory applies to the arboreal gas, it remains open to
prove even the special case of \emph{negative correlation}: for
distinct edges $e,f\in E$,
\begin{equation}
  \label{eq:NC}
  \mathbb{P}_{\beta}[e,f\in F] \leq \mathbb{P}_{\beta}[e\in F]\mathbb{P}_{\beta}[f\in F].
\end{equation}
Negative correlation for all weights is equivalent to all connection
probabilities
$\mathbb{P}_{\beta}[0\leftrightarrow j]=\langle
e^{t_j}\rangle_\beta^0$ being increasing in all weights
$\beta=(\beta_{ij})$, where the right-hand side is in terms of the
$t$-field of the pinned $\mathbb{H}^{0|2}$ model.  The analogue for
the $\mathbb{H}^{2|2}$ model is precisely Poudevigne's inequality,
Theorem~\ref{thm:poudevigne}. Does this inequality extend to other
$\mathbb{H}^{n|2m}$ models?

\section{Concluding remarks}
\label{sec:concluding-remarks}
This survey has focused on the connections between hyperbolic spin
systems and probabilistic models that share phenomenology with the
Anderson transition, including a number of open questions. It is also
worth repeating a question from~\cite{MR3469136}: are there other models
of random walk that are related to spin systems? A partial answer was
given in~\cite{MR4255180}, but we expect there is more to be discovered, see,
e.g.,~\cite{2102.08988}. Similarly, one may search for probabilistic
representations of $\mathbb{H}^{n|2m}$ models for values of $n,m$ not discussed
here.


\begin{ack}
  Special thanks are due to Nick Crawford and Andrew Swan for their
  collaboration~\cite{MR4021254,MR4255180,MR4218682,2107.01878} on the
  papers on which much of this survey is based.  We also thank
  R\'{e}my Poudevigne 
  for explaining the origin of the STZ
  field as discussed in Section~\ref{sec:STZ}.
\end{ack}

\begin{funding}
  This work was partially supported by the European Research Council
  under the European Union's Horizon 2020 research and innovation
  programme (grant No.~851682 SPINRG).
\end{funding}

\bibliographystyle{emss}

\begin{thebibliography}{10}
\providecommand{\url}[1]{\texttt{#1}}
\providecommand{\urlprefix}{URL }
\providecommand{\eprint}[2][]{\url{#2}}

\bibitem{MR3364516}
M.~Aizenman and S.~Warzel, \emph{Random operators}. Graduate Studies in
  Mathematics 168, American Mathematical Society, Providence, RI, 2015.

\bibitem{MR3524748}
N.~Alon and J.~Spencer, \emph{The probabilistic method}. Fourth edn., Wiley
  Series in Discrete Mathematics and Optimization, 2016.

\bibitem{1811.01600}
N.~Anari, K.~Liu, S.~Gharan, and C.~Vinzant, Log-concave polynomials {III}:
  Mason's ultra-log-concavity conjecture for independent sets of matroids.
  arXiv:1811.01600.

\bibitem{MR3189433}
O.~Angel, N.~Crawford, and G.~Kozma, Localization for linearly edge reinforced
  random walks. \emph{Duke Math. J.} \textbf{163} (2014), no.~5,
  889--921.

\bibitem{MR2816340}
S.~Bacallado, Bayesian analysis of variable-order, reversible {M}arkov chains.
  \emph{Ann. Statist.} \textbf{39} (2011), no.~2, 838--864.

\bibitem{MR3534357}
S.~Bacallado, V.~Pande, S.~Favaro, and L.~Trippa, Bayesian regularization of
  the length of memory in reversible sequences. \emph{J. R. Stat. Soc. Ser. B.
  Stat. Methodol.} \textbf{78} (2016), no.~4, 933--946.

\bibitem{MR2985176}
A.-L. Basdevant and A.~Singh, Continuous-time vertex reinforced jump processes
  on {G}alton-{W}atson trees. \emph{Ann. Appl. Probab.} \textbf{22} (2012),
  no.~4, 1728--1743.

\bibitem{MR3345374}
R.~Bauerschmidt, D.~Brydges, and G.~Slade, Critical two-point function of the
  4-di\-men\-sion\-al weakly self-avoiding walk. \emph{Commun. Math. Phys.}
  \textbf{338} (2015), no.~1, 169--193.

\bibitem{MR3339164}
R.~Bauerschmidt, D.~Brydges, and G.~Slade, Logarithmic correction for the
  susceptibility of the 4-dimensional weakly self-avoiding walk: a
  renormalisation group analysis. \emph{Commun. Math. Phys.} \textbf{337}
  (2015), no.~2, 817--877.

\bibitem{2107.01878}
R.~Bauerschmidt, N.~Crawford, and T.~Helmuth, Percolation transition for random
  forests in $d \geq 3$  (2021).

\bibitem{MR4218682}
R.~Bauerschmidt, N.~Crawford, T.~Helmuth, and A.~Swan, Random {S}panning
  {F}orests and {H}yperbolic {S}ymmetry. \emph{Commun. Math. Phys.}
  \textbf{381} (2021), no.~3, 1223--1261.

\bibitem{MR4021254}
R.~Bauerschmidt, T.~Helmuth, and A.~Swan, Dynkin isomorphism and
  {M}ermin--{W}agner theorems for hyperbolic sigma models and recurrence of the
  two-dimensional vertex-reinforced jump process. \emph{Ann. Probab.}
  \textbf{47} (2019), no.~5, 3375--3396.

\bibitem{MR4255180}
R.~Bauerschmidt, T.~Helmuth, and A.~Swan, The geometry of random walk
  isomorphism theorems. \emph{Ann. Inst. Henri Poincar\'{e} Probab. Stat.}
  \textbf{57} (2021), no.~1, 408--454.

\bibitem{MR2948685}
V.~Beffara and H.~Duminil-Copin, The self-dual point of the two-dimensional
  random-cluster model is critical for {$q\geq 1$}. \emph{Probab. Theory
  Related Fields} \textbf{153} (2012), no. 3-4, 511--542.

\bibitem{MR914369}
F.~Berezin, \emph{Introduction to superanalysis}. Mathematical Physics and
  Applied Mathematics 9, 1987.

\bibitem{MR2476782}
J.~Borcea, P.~Br\"{a}nd\'{e}n, and T.~Liggett, Negative dependence and the
  geometry of polynomials. \emph{J. Amer. Math. Soc.} \textbf{22} (2009),
  no.~2, 521--567.

\bibitem{MR4172622}
P.~Br\"{a}nd\'{e}n and J.~Huh, Lorentzian polynomials. \emph{Ann. of Math. (2)}
  \textbf{192} (2020), no.~3, 821--891.

\bibitem{MR2523458}
D.~Brydges, Lectures on the renormalisation group. In \emph{Statistical
  mechanics}, pp. 7--93, IAS/Park City Math. Ser. 16, Amer. Math. Soc.,
  2009.

\bibitem{MR648362}
D.~Brydges, J.~Fr{\"o}hlich, and T.~Spencer, The random walk representation of
  classical spin systems and correlation inequalities. \emph{Commun. Math.
  Phys.} \textbf{83} (1982), no.~1, 123--150.

\bibitem{MR3332938}
D.~Brydges and G.~Slade, A renormalisation group method. {I}. {G}aussian
  integration and normed algebras. \emph{J. Stat. Phys.} \textbf{159} (2015),
  no.~3, 421--460.

\bibitem{MR3332939}
D.~Brydges and G.~Slade, A renormalisation group method. {II}. {A}pproximation
  by local polynomials. \emph{J. Stat. Phys.} \textbf{159} (2015), no.~3,
  461--491.

\bibitem{MR2110547}
S.~Caracciolo, J.~Jacobsen, H.~Saleur, A.~Sokal, and A.~Sportiello, Fermionic
  field theory for trees and forests. \emph{Phys. Rev. Lett.} \textbf{aracc}
  (2004), no.~8, 080601, 4.

\bibitem{1911.09036}
Y.~Chang, D.-Z. Liu, and X.~Zeng, On $H^{2|2}$ isomorphism theorems and
  reinforced loop soup, arXiv:1911.09036.

\bibitem{MR1048927}
J.~Chayes, L.~Chayes, G.~Grimmett, H.~Kesten, and R.~Schonmann, The correlation
  length for the high-density phase of {B}ernoulli percolation. \emph{Ann.
  Probab.} \textbf{17} (1989), no.~4, 1277--1302.

\bibitem{1912.05817}
N.~Crawford, Supersymmetric hyperbolic $\sigma$-models and decay of
  correlations in two dimensions. arXiv:1912.05817.

\bibitem{MR1900324}
B.~Davis and S.~Volkov, Continuous time vertex-reinforced jump processes.
  \emph{Probab. Theory Related Fields} \textbf{123} (2002), no.~2, 281--300.

\bibitem{MR2027294}
B.~Davis and S.~Volkov, Vertex-reinforced jump processes on trees and finite
  graphs. \emph{Probab. Theory Related Fields} \textbf{128} (2004), no.~1,
  42--62.

\bibitem{PhysRevLett.98.030602}
Y.~Deng, T.~Garoni, and A.~Sokal, Ferromagnetic phase transition for the
  spanning-forest model ($q\ensuremath{\rightarrow}0$ limit of the potts model)
  in three or more dimensions. \emph{Phys. Rev. Lett.} \textbf{98} (2007),
  030602.

\bibitem{MR1008047}
P.~Diaconis, Recent progress on de {F}inetti's notions of exchangeability. In
  \emph{Bayesian statistics, 3 ({V}alencia, 1987)}, pp. 111--125, Oxford Sci.
  Publ., 1988.

\bibitem{MR577313}
  P.~Diaconis and D.~Freedman, Finite exchangeable sequences. \emph{Ann. Probab.}
  \textbf{8} (1980), no.~4, 745--764.

\bibitem{MR2278358}
P.~Diaconis and S.~Rolles, Bayesian analysis for reversible {M}arkov chains.
  \emph{Ann. Statist.} \textbf{34} (2006), no.~3, 1270--1292.

\bibitem{MR3366053}
M.~Disertori, C.~Sabot, and P.~Tarr\`es, Transience of edge-reinforced random
  walk. \emph{Commun. Math. Phys.} \textbf{339} (2015), no.~1, 121--148.

\bibitem{MR2736958}
M.~Disertori and T.~Spencer, Anderson localization for a supersymmetric sigma
  model. \emph{Commun. Math. Phys.} \textbf{300} (2010), no.~3, 659--671.

\bibitem{MR2728731}
M.~Disertori, T.~Spencer, and M.~Zirnbauer, Quasi-diffusion in a 3{D}
  supersymmetric hyperbolic sigma model. \emph{Commun. Math. Phys.}
  \textbf{300} (2010), no.~2, 435--486.

\bibitem{MR3765895}
H.~Duminil-Copin, A.~Raoufi, and V.~Tassion, A new computation of the critical
  point for the planar random-cluster model with {$q\ge1$}. \emph{Ann. Inst.
  Henri Poincar\'{e} Probab. Stat.} \textbf{54} (2018), no.~1, 422--436.

\bibitem{PhysRevB.54.12763}
T.~Dupr{\'e}, Localization transition in three dimensions: Monte carlo
  simulation of a nonlinear \ensuremath{\sigma} model. \emph{Phys. Rev. B}
  \textbf{54} (1996), 12763--12774.

\bibitem{MR693227}
E.~Dynkin, Markov processes as a tool in field theory. \emph{J. Funct. Anal.}
  \textbf{50} (1983), no.~2, 167--187.

\bibitem{2108.04335}
P.~Easo, The wired arboreal gas on regular trees. arXiv:2108.04335.

\bibitem{MR708812}
K.~Efetov, Supersymmetry and theory of disordered metals. \emph{Adv. in Phys.}
  \textbf{32} (1983), no.~1, 53--127.

\bibitem{MR359655}
C.~Fortuin and P.~Kasteleyn, On the random-cluster model. {I}. {I}ntroduction
  and relation to other models. \emph{Physica} \textbf{57} (1972), 536--564.

\bibitem{MR0309498}
C.~Fortuin, P.~Kasteleyn, and J.~Ginibre, Correlation inequalities on some
  partially ordered sets. \emph{Commun. Math. Phys.} \textbf{22} (1971),
  89--103.

\bibitem{MR3451159}
A.~Goel, S.~Khanna, S.~Raghvendra, and H.~Zhang, Connectivity in random forests
  and credit networks. In \emph{Proceedings of the {T}wenty-{S}ixth {A}nnual
  {ACM}-{SIAM} {S}ymposium on {D}iscrete {A}lgorithms}, pp. 2037--2048, 2015.

\bibitem{MR2243761}
G.~Grimmett, \emph{The random-cluster model}. Grundlehren der Mathematischen
  Wissenschaften 333, Springer-Verlag, Berlin, 2006.

\bibitem{MR3729454}
M.~Heydenreich and R.~van~der Hofstad, \emph{Progress in high-dimensional
  percolation and random graphs}. CRM Short Courses, Springer, 2017.

\bibitem{MR2142209}
J.L.~Jacobsen and H.~Saleur,
The arboreal gas and the supersphere sigma model.
\emph{Nuclear Phys. B} \textbf{716(3)} (2005), 439--461.

\bibitem{MR2724667}
J.~Kahn and M.~Neiman, Negative correlation and log-concavity. \emph{Random
  Structures Algorithms} \textbf{37} (2010), no.~3, 367--388.

\bibitem{MR1832379}
M.~Keane and S.~Rolles, Edge-reinforced random walk on finite graphs. In
  \emph{Infinite dimensional stochastic analysis ({A}msterdam, 1999)}, pp.
  217--234, Verh. Afd. Natuurkd. 1. Reeks. K. Ned. Akad. Wet. 52, R. Neth.
  Acad. Arts Sci., Amsterdam, 2000.

\bibitem{MR3469136}
G.~Kozma, Reinforced random walk. In \emph{European {C}ongress of
  {M}athematics}, pp. 429--443, Eur. Math. Soc., Z\"{u}rich, 2013.

\bibitem{1911.08579}
G.~Kozma and R.~Peled, Power-law decay of weights and recurrence of the
  two-dimensional {VRJP}. Electron. J. Probab. 26 (2021), no.~82.

\bibitem{MR1167294}
T.~{\L}uczak and B.~Pittel, Components of random forests. \emph{Combin. Probab.
  Comput.} \textbf{1} (1992), no.~1, 35--52.

\bibitem{MR3616205}
R.~Lyons and Y.~Peres, \emph{Probability on trees and networks}. Cambridge
  Series in Statistical and Probabilistic Mathematics 42, 2016.

\bibitem{MR2250510}
M.~Marcus and J.~Rosen, \emph{Markov processes, {G}aussian processes, and local
  times}. Cambridge Studies in Advanced Mathematics 100, 2006.

\bibitem{MR3845513}
J.~Martin and D.~Yeo, Critical random forests. \emph{ALEA Lat. Am. J. Probab.
  Math. Stat.} \textbf{15} (2018), no.~2, 913--960.

\bibitem{MR2441859}
F.~Merkl, A.~\"Ory, and S.~Rolles, The `magic formula' for linearly
  edge-reinforced random walks. \emph{Statist. Neerlandica} \textbf{62} (2008),
  no.~3, 345--363.

\bibitem{MR2561431}
F.~Merkl and S.~Rolles, Recurrence of edge-reinforced random walk on a
  two-dimensional graph. \emph{Ann. Probab.} \textbf{37} (2009), no.~5,
  1679--1714.

\bibitem{MR1843511}
A.~Mirlin, Statistics of energy levels and eigenfunctions in disordered and
  chaotic systems: supersymmetry approach. In \emph{New directions in quantum
  chaos ({V}illa {M}onastero, 1999)}, pp. 223--298, Proc. Internat. School
  Phys. Enrico Fermi 143, IOS, Amsterdam, 2000.

\bibitem{MR1127715}
R.~Pemantle, Choosing a spanning tree for the integer lattice uniformly.
  \emph{Ann. Probab.} \textbf{19} (1991), no.~4, 1559--1574.

\bibitem{MR1757964}
R.~Pemantle, Towards a theory of negative dependence. pp. 1371--1390, 41, 2000.

\bibitem{MR2282181}
R.~Pemantle, A survey of random processes with reinforcement. \emph{Probab.
  Surv.} \textbf{4} (2007), 1--79.

\bibitem{1911.02181}
R.~Poudevigne, Monotonicity and phase transition for the {VRJP} and the {ERRW}.
  arXiv:1911.02181.

\bibitem{2108.04287}
G.~Ray and B.~Xiao, Forests on wired regular trees. arXiv:2108.04287.

\bibitem{MR4218029}
C.~Sabot, Polynomial localization of the 2{D}-vertex reinforced jump process.
  \emph{Electron. Commun. Probab.} \textbf{26} (2021), Paper No. 1, 9.

\bibitem{MR3420510}
C.~Sabot and P.~Tarr\`es, Edge-reinforced random walk, vertex-reinforced jump
  process and the supersymmetric hyperbolic sigma model. \emph{J. Eur. Math.
  Soc.} \textbf{17} (2015), no.~9, 2353--2378.

\bibitem{2102.08988}
C.~Sabot and P.~Tarr{\`e}s, The $*$-{V}ertex-{R}einforced {J}ump {P}rocess.
  arXiv:2102.08988.

\bibitem{MR3729620}
C.~Sabot, P.~Tarr\`es, and X.~Zeng, The vertex reinforced jump process and a
  random {S}chr\"odinger operator on finite graphs. \emph{Ann. Probab.}
  \textbf{45} (2017), no.~6A, 3967--3986.

\bibitem{MR3904155}
C.~Sabot and X.~Zeng, A random {S}chr\"{o}dinger operator associated with the
  {V}ertex {R}einforced {J}ump {P}rocess on infinite graphs. \emph{J. Amer.
  Math. Soc.} \textbf{32} (2019), no.~2, 311--349.

\bibitem{MR575503}
L.~Sch\"{a}fer and F.~Wegner, Disordered system with {$n$} orbitals per site:
  {L}agrange formulation, hyperbolic symmetry, and {G}oldstone modes. \emph{Z.
  Phys. B} \textbf{38} (1980), no.~2, 113--126.

\bibitem{MR2953867}
T.~Spencer, S{USY} statistical mechanics and random band matrices. In
  \emph{Quantum many body systems}, pp. 125--177, Lecture Notes in Math. 2051,
  Springer, Heidelberg, 2012.

\bibitem{MR3204347}
T.~Spencer, Duality, statistical mechanics, and random matrices. In
  \emph{Current developments in mathematics 2012}, pp. 229--260, 2013.

\bibitem{MR2104878}
T.~Spencer and M.~Zirnbauer, Spontaneous symmetry breaking of a hyperbolic
  sigma model in three dimensions. \emph{Commun. Math. Phys.} \textbf{252}
  (2004), no. 1-3, 167--187.

\bibitem{Syma69}
K.~Symanzik, Euclidean quantum field theory. In \emph{Local quantum field
  theory}, edited by R.~Jost, Academic Press, New York, 1969.

\bibitem{10.1007/BF01319839}
F.~Wegner, The mobility edge problem: {C}ontinuous symmetry and a conjecture.
  \emph{Z Physik B} \textbf{35} (1979), 207--210.

\bibitem{MR863830}
M.~Zirnbauer, Localization transition on the {B}ethe lattice. \emph{Phys. Rev.
  B (3)} \textbf{34} (1986), no.~9, 6394--6408.

\bibitem{MR1134935}
M.~Zirnbauer, Fourier analysis on a hyperbolic supermanifold with constant
  curvature. \emph{Commun. Math. Phys.} \textbf{141} (1991), no.~3, 503--522.

\end{thebibliography}

\end{document}